\documentclass{amsart}
\usepackage[margin=2.5cm]{geometry}
\usepackage{graphics, amsmath,amssymb, enumitem, comment, url, mathtools}
\usepackage{graphicx}
\usepackage{epsfig}
\usepackage{tikz}
\usepackage{forest}
\usetikzlibrary{trees}

\newif\ifcolorcomments
\newcommand{\allowcomments}[4]{
\newcommand{#1}[1]{\ifdraft{\ifcolorcomments{\textcolor{#4}{##1 --#3}}\else{\textsl{ ##1 \ --#3}}\fi}\else{}\fi}
}

\colorcommentstrue
\usepackage{amssymb}
\usepackage{amsmath,amsthm}
\usepackage{hyperref}

\usepackage{colortbl}

\newtheorem{theorem}{Theorem}
\newtheorem{lemma}[theorem]{Lemma}
\newtheorem{proposition}[theorem]{Proposition}

\theoremstyle{definition}

\newtheorem{remark}[theorem]{Remark}

\newtheorem{manualtheoreminner}{Theorem}
\newenvironment{manualtheorem}[1]{%
  \renewcommand\themanualtheoreminner{#1}%
  \manualtheoreminner
}{\endmanualtheoreminner}

\newcommand{\N}{\mathbb N}

\newcommand{\R}{\mathbb R}

\newcommand{\Z}{\mathbb Z}

\renewcommand{\text}{\textup}

\newcommand{\NPC}[1]{\ignorespaces}

\newif\ifdraft\drafttrue

\def\N{\mathbb N}
\def\Z{\mathbb Z}

\def\R{\mathbb R}

\IfFileExists{marvosym.sty}{
\RequirePackage{marvosym} 
}{

}

\IfFileExists{wasysym.sty}{
\RequirePackage{wasysym}
}{

}

\newcommand*{\myDots}{\ifmmode\mathellipsis\else.\kern-0.07em.\kern-0.07em.\fi}

\usepackage{enumitem}

\allowcomments{\comvika}{VR}{Vika}{green}
\allowcomments{\comnikita}{NS}{Nikita}{blue}

\newcommand {\ignore}[1] {}

\usepackage{mathtools}

\usepackage{ longtable}

\begin{document}

\title{Difference of irrationality measure functions}

\author[Viktoria Rudykh]{Viktoria Rudykh}
\address{Viktoria Rudykh,  Faculty of Mathematics, Technion, Haifa, Israel.}
\email{rudykh.vt@gmail.com}

\author{Nikita Shulga}
\address{Nikita Shulga, Institute for Problems of Information Transmission, Moscow, Russia and Moscow Center for Fundamental and Applied Mathematics, Moscow, Russia.}
\email{nikos1279@gmail.com}
\date{}

\maketitle

\begin{abstract}
For an irrational number $\alpha\in\mathbb{R}$ we consider its irrationality measure function $$ \psi_\alpha(x) = \min_{1\le q\le x,\, q\in\mathbb{Z}} \| q\alpha \|. $$ It is known for all irrational numbers $\alpha$ and $\beta$ satisfying $\alpha\pm\beta\not\in\Z$, there exist arbitrary large values of $t$ with
\begin{equation*}
	| \psi_\alpha(t) - \psi_\beta(t) | \geqslant \left( \sqrt{\tau} - 1\right) \cdot \min( \psi_\alpha(t), \psi_\beta(t) ),
\end{equation*}
where $\tau = \frac{\sqrt{5} + 1}{2}$ and this result is optimal for certain numbers equivalent to $\tau$. Here we prove that for all irrational numbers $\alpha$ and $\beta$, satisfying $\alpha\pm\beta\not\in\Z$, such that at least one of them is not equivalent to $\tau$, there exist arbitrary large values of $t$ with
$$  | \psi_\alpha(t) - \psi_\beta(t) | \geqslant (\sqrt{\sqrt2+1}-1)\cdot \min( \psi_\alpha(t), \psi_\beta(t) ). $$
Moreover, we show that the constant on the right-hand side is optimal.

\end{abstract}

\section{Introduction}

For an irrational number $\alpha\in\mathbb{R}$ we consider its irrationality measure function
$$ \psi_\alpha(x) = \min_{1\le q\le x,\, q\in\mathbb{Z}} \| q\alpha \| ,$$
where $||\cdot || $ denotes the distance to the nearest integer.\\
Let
$$ q_0\le q_1<q_2<\ldots<q_n<q_{n+1}<\ldots $$
be the sequence of the denominators of convergents to $\alpha$. It is a well-known fact (see \cite{Perron} or \cite{Schmidt}) that
\begin{equation}\label{minimum}
 \psi_\alpha(x) = \| q_n\alpha \| \text{  for  } q_n\le x < q_{n+1}. 
 \end{equation}
It is clear from Dirichlet's theorem that for any $\alpha\in\mathbb{R}$ and for every $x\ge1$ we have
\begin{equation*}\label{dirichlet}
\psi_\alpha(x) \le \frac{1}{x}. 
 \end{equation*}
 
In \cite{KM}, Kan and Moshchevitin proved that for any two irrational numbers $\alpha, \beta$ satisfying $\alpha\pm\beta\notin\mathbb{Z}$, the difference 
$$ \psi_\alpha(t) - \psi_\beta(t)  $$
changes its sign infinitely many times as $t\to\infty$.

In 2017 Dubickas obtained a slightly more general statement using the language of combinatorics on words. Namely, he proved
\begin{theorem}\label{dubic}
Let $\alpha$ and $\beta$ be two irrational real numbers satisfying $\alpha\pm\beta\not\in\Z$. Then
$$
\psi_\alpha(n)> ||n\beta ||
$$
for infinitely many $n\in\N$.
\end{theorem}
By swapping roles of $\alpha$ and $\beta$ in Theorem \ref{dubic} and noting that $\psi_\beta(n)\leq ||n\beta||$, one easily gets the result of Kan and Moshchevitin. For more results on this topic via combinatorics on words, see \cite{Dubickas}.

In 2019 Moshchevitin \cite{M} proved the following result on a difference of irrationality measure functions.

\begin{theorem}\label{TheoremBB}		Let $\alpha$ and $\beta$ be the two irrational numbers satisfying $\alpha\pm\beta\notin\mathbb{Z}$. Then for every $T\ge 1$ there exists $t\ge T$ such that

\begin{equation}\label{mosh11}
 | \psi_\alpha(t) - \psi_\beta(t) | \geqslant C_1 \cdot \min( \psi_\alpha(t), \psi_\beta(t) ), 
\end{equation}
where $C_1=\sqrt{\frac{\sqrt5+1}{2}}-1\approx0.272^+$.
\end{theorem}
It was also shown that constant $C_1$ in \eqref{mosh11} is optimal: author provided an example of two specific numbers $\alpha$ and $\beta$, both equivalent to golden ratio, such that the constant $C_1$ in \eqref{mosh11} can not be improved. The notion of equivalence is given in \eqref{defequiv}.

For a pair of irrational numbers $\alpha, \beta$ satisfying $\alpha\pm\beta\notin\mathbb{Z}$, define $C_{\alpha,\beta}$ as 
$$
C_{\alpha,\beta} = \sup \{ C:  | \psi_\alpha(t) - \psi_\beta(t) | \geqslant C \cdot \min( \psi_\alpha(t), \psi_\beta(t) )  \text {\,\, for infinitely many $t\in\N$} \}.
$$
Using this notation, Theorem \ref{TheoremBB} can be reformulated as follows.
\begin{manualtheorem}{\ref{TheoremBB}'}\label{TheoremB}		Let $\alpha$ and $\beta$ be the two irrational numbers satisfying $\alpha\pm\beta\notin\mathbb{Z}$. Then 
\begin{equation*}\label{mosh1}
C_{\alpha,\beta} \geq C_1,
\end{equation*}
where $C_1=\sqrt{\frac{\sqrt5+1}{2}}-1\approx0.272^+$.
\end{manualtheorem}
In new notation the statement about the optimality of $C_1$ means that there exist two specific numbers $\alpha$ and $\beta$, both equivalent to golden ratio, such that $C_{\alpha,\beta} =C_1$. It turns out that this example is in some sense unique. If we assume that one of the numbers is not equivalent to the golden ratio, then constant $C_1$ can be significantly improved. To be more precise, the main result of present paper is 
\begin{theorem}\label{maintheorem}
Let $\alpha$ and $\beta$ be the two irrational numbers satisfying $\alpha\pm\beta\notin\mathbb{Z}$. If at least one of the numbers $\alpha$ or $\beta$ is not equivalent to $\tau=\frac{\sqrt5+1}{2}$, then
\begin{equation}\label{mainineq}
C_{\alpha,\beta} \geq C_2,
\end{equation}
 where $C_2=\sqrt{\sqrt2+1}-1 \approx0.5537^+$.
\end{theorem}
We also prove that the constant $C_2$ in the previous theorem can not be improved.
\begin{theorem}\label{optimality}
The constant $C_2$ in Theorem \ref{maintheorem} is optimal in a sense that there exist two irrational numbers $\theta$ and $\omega$, such that $\theta\pm\omega\notin\Z$, at least of them is not equivalent to $\tau$ and
$$
C_{\theta,\omega} =  C_2.
$$
\end{theorem}

The paper is organised as follows. In Section \ref{sec2} we provide some auxiliary results from the theory of continued fractions and introduce necessary construction from combinatorics on words. In Section \ref{sec3} we prove Theorem \ref{maintheorem}. Section \ref{sec4} is devoted to the proof of Theorem \ref{optimality}. Finally, in Section \ref{sec5} we provide some complimentary results about the value of $C_{\alpha,\beta}$ for some specific cases of $\alpha$ and $\beta$.

\section{Preliminaries}\label{sec2}
Throughout, we assume that $\alpha\notin\mathbb{Q}$ and we denote its continued fraction expansion as
\begin{equation*}
\alpha =  a_0 + \cfrac{1}{a_1+\cfrac{1}{a_2+\cdots}}=[a_0;a_1,a_2,\ldots], \,\,\, a_j \in \mathbb{Z}_+.
\end{equation*}
By $\alpha_r$ we denote the tail of continued fraction
\begin{equation}\label{taila}
\alpha_r=  [a_r;a_{r+1},a_{r+2},\ldots].
\end{equation}
Similarly, let $$\beta = [b_0; b_1,b_2,\ldots]$$ and $$\beta_m = [b_m; b_{m+1}, b_{m+2}, \ldots].$$
We call two irrational numbers $\alpha$ and $\beta$ \textit{equivalent} and write
\begin{equation}\label{defequiv}
\alpha\sim_c\beta
\end{equation}
 if there exist $n,m\in\N$, such that $\alpha_n=\beta_m$. We denote the equivalence as $\alpha\sim_c\beta$ to distinguish it from asymptotic equivalence $\sim$. 

For irrational numbers $\alpha$ and $\beta$ we denote the denominators of their convergents by $q_n$ and $t_m$ respectively.\\
Define 
$$\xi_n = \psi_\alpha( q_n ) \text{\,\,\,\,\, and \,\,\,\,\,\,} \eta_s = \psi_\beta ( t_s ).$$
It is well-known that 
\begin{equation*}
q_n\alpha - \frac{(-1)^n}{q_n\alpha_{n+1}+q_{n-1}}\in\mathbb{Z}.
\end{equation*}
Using this fact as well as \eqref{minimum}, we see that

\begin{equation}\label{defal}
\psi_\alpha(n) = \frac{1}{q_r\alpha_{r+1}+q_{r-1}}=\frac{1}{q_{r+1}+\frac{q_r}{\alpha_{r+2}}},
\end{equation}
where  $r\geqslant0$ is the largest integer satisfying $q_r\leqslant n$.
Similarly,
\begin{equation*}\label{defbe}
\psi_\beta(n) = \frac{1}{t_l\beta_{l+1}+t_{l-1}}=\frac{1}{t_{l+1}+\frac{t_l}{\beta_{l+2}}},
\end{equation*}
where  $l\geqslant0$ is the largest integer satisfying $t_l\leqslant n$.\\
From the definition of $\xi_n$ and \eqref{defal}, we immediately see that
\begin{equation}\label{tailxi}
 \frac{ \xi_{n-1}}{\xi_{n}} = \alpha_{n+1} .
\end{equation}
For the proof of our results, we use language of combinatorics on words. In connection with irrationality measure functions it was introduced by Dubickas in \cite{Dubickas} and was extensively used by the second author in \cite{Shulga}.\\
 Consider a union $U=D_{\alpha} \cup D_{\beta}$ of two sequences $D_\alpha = \{ q_0 = 1 \le q_1 < q_2 < \ldots \}$ and $D_\beta = \{ t_0 = 1 \le t_1 < t_2 < \ldots \} $ of denominators of convergents of $\alpha$ and $\beta$. We construct an infinite word $W$ on an alphabet $\{ B^*_*, Q^*, T_* \}$, where $*$ will match some indices of elements from $D_\alpha$ and $D_\beta$, by the following procedure. The first element of $U$ is $1 = q_0 = t_0$. It belongs to both sequences, so we start our infinite word with $B^0_0$. Then, we are looking at the next element of $U$. There are three potential situations:
\begin{enumerate}
 \item If the next element of $U$ is $q_j$ from the set $D_\alpha$ and it does not belong to the set $D_\beta$, we write the letter $Q^j$; \\
 \item If the next element of $U$ is $t_i$ from the set $D_\beta$ and it does not belong to the set $D_\alpha$, we write the letter $T_i$; \\
 \item If the next element of $U$ is equal to $q_j$ from $D_\alpha$ and also to $t_i$ from $D_\beta$, we write the letter $B^j_i$.
 \end{enumerate} Sometimes we will omit indices of letters $B,Q$ and $T$ if this does not cause ambiguity.

From now on we only consider irrational numbers $\alpha$ and $\beta$ with $\alpha\pm\beta\notin\mathbb{Z}$.  This condition on $\alpha$ and $\beta$ implies two simple lemmas, reformulated in the word notation.
\begin{lemma}[\cite{M}, Hilfssatz 3.i]\label{lemma1}
There exist only finitely many subwords $BB$ in a word $W$.
\end{lemma}

\begin{lemma}[\cite{M}, Hilfssatz 3.ii]\label{lemma2}
There exist only finitely many subwords $B^{n-1}_{s-1}Q^{n}B^{n+1}_{s}$ with $a_{n+1}=1$ in a word $W$.
\end{lemma}
We will also make use of the following statement.
\begin{lemma}\label{lemma3}
If $a_n=2$, then 
$$
\min ( \alpha_{n}, \alpha_{n+1} ) \geq \sqrt2+1.
$$
\end{lemma}
\begin{proof}
By definition \eqref{taila} of $\alpha_n$, we get
$$
\alpha_n = a_n +\frac{1}{\alpha_{n+1}} = 2+ \frac{1}{\alpha_{n+1}} .
$$
Recall that $(\sqrt2-1)(\sqrt2+1)=1$. Now \\
If $\alpha_{n+1}<\sqrt2+1$, then $\alpha_n>2+\frac{1}{\sqrt2+1}=\sqrt2+1$.\\
If $\alpha_n<\sqrt2+1$, then $\frac{1}{\alpha_{n+1}}< \sqrt2-1$, or $\alpha_{n+1}> \frac{1}{\sqrt2-1}=\sqrt2+1$.
\end{proof}
Next lemma was introduced in \cite{KM} by Kan and Moshchevitin.
\begin{lemma}\label{lemma4}
Suppose that
$$
s,n \geq 2 \,\,\, \text{ and} \,\,\,  q_{n+1}\leq t_{s}.
$$
Then 
$$
\xi_{n-1} > \eta_{s-1}.
$$
\end{lemma}

Finally, following lemma was established by Moshchevitin in \cite{M}, we reformulate it using our word notation.
\begin{lemma}[\cite{M}, Folgerung 1]\label{lemma6}
Suppose there is a letter $Q^r$ in the word $W$ and $\alpha_{r+1}\geq C$ for some $C$, then we have
\begin{align*}
either \,\,\,\,\,\,\,  | \psi_\alpha(t) - \psi_\beta(t) | \geqslant (\sqrt{C}-1) \cdot \min( \psi_\alpha(t), \psi_\beta(t) )   \,\,\,\,  for \,\,\, t\in [q_r-1,q_r), \\
or \,\,\,\,  | \psi_\alpha(q_r) - \psi_\beta(q_r) | \geqslant (\sqrt{C}-1) \cdot \min( \psi_\alpha(q_r), \psi_\beta(q_r) ). 
\end{align*}
\end{lemma}

Let us also denote a constant
$$
K=C_2+1 = \sqrt{\sqrt2+1}\approx 1.553^+,
$$
as it will be extensively used throughout the proof.\\

\section{Proof of Theorem \ref{maintheorem}}\label{sec3}
 Suppose that the statement of Theorem \ref{maintheorem} is false. Then by definition of $C_{\alpha,\beta}$, for some $\alpha$ and $\beta$ with $\alpha \pm \beta \not\in \Z$ and at least one of the numbers $\alpha$ or $\beta$ is not equivalent to $\tau=\frac{\sqrt5+1}{2}$, there exists $T_0 \in \R$ such that
\begin{equation}
	\label{eq1}
 	| \psi_{\alpha}(t) - \psi_{\beta}(t)| < C_2 \min(\psi_{\alpha}(t), \psi_{\beta}(t)) \,\,\,\,\,\,\,\,\, \text{for all  } t\geq T_0.
\end{equation}
All of the propositions in this section will be proved under the assumption \eqref{eq1}.

\begin{proposition}\label{prop1}
Suppose there is a letter $Q^n$  in the word $W$, then 
\begin{enumerate}[label=(\roman*)]
\item
If $\xi_n<\psi_\beta(q_n) <\xi_{n-1}$, then $\alpha_{n+1}\leq K^2$, in particular $a_{n+1}\in\{1,2\}$ and a pair $(a_{n+1},a_{n+2})\neq(2,1)$.  \\
\item 
If $\xi_{n-1} \leqslant \psi_\beta(q_n)$  or $\psi_\beta(q_n) \leqslant \xi_{n}$, then $\alpha_{n+1}\leq K$, in particular $a_{n+1}=1$.
\end{enumerate}
\end{proposition}
\begin{proof}
In the case (i) from \eqref{eq1} we get 
$$
\xi_{n-1} - \psi_\beta(q_n) < C_2 \psi_\beta(q_n)  \implies \xi_{n-1} < K \psi_\beta(q_n) 
$$
and
$$
\psi_\beta(q_n)  - \xi_{n}< C_2\xi_{n}  \implies \psi_\beta(q_n)  < K \xi_{n}.
$$
Combining two inequalities together and using \eqref{tailxi}, we come to
$$
\xi_{n-1}<K^2 \xi_{n}  \implies \frac{\xi_{n-1}}{\xi_{n}} < K^2 \implies \alpha_{n+1} < K^2= \sqrt2+1 \approx 2.414^+.
$$
As a consequence we see that $a_{n+1}\in\{1,2\}$ and that $\alpha_{n+1}$ can not start as $\alpha_{n+1} = [2;1,\ldots]$. \\
 In the case (ii) if $\xi_{n-1} \leqslant \psi_\beta(q_n)$, then for $0<\varepsilon<1$ we have
$$
C_2 \xi_n > \psi_{\beta}(q_n+\varepsilon) - \psi_\alpha(q_n+\varepsilon) \geqslant \xi_{n-1}-\xi_n = (\alpha_{n+1}-1)\xi_n.
$$
If $\psi_\beta(q_n) \leqslant \xi_{n}$, then
$$
C_2 \psi_{\beta}(q_n-\varepsilon)>   \psi_{\alpha}(q_n-\varepsilon) - \psi_{\beta}(q_n-\varepsilon) \geqslant \xi_{n-1} - \xi_n =  (\alpha_{n+1}-1)\xi_n >  (\alpha_{n+1}-1) \psi_{\beta}(q_n-\varepsilon).
$$
In both inequalities we get $\alpha_{n+1} < K$, which obviously implies $a_{n+1}=1$.
\end{proof}
From Proposition \ref{prop1}, Lemma \ref{lemma3} and Lemma \ref{lemma6}, we get one important corollary.
\begin{proposition}\label{prop2}
If the word $W$ has a subword $B^{n-1}_{s-1} \dots B^{n+k}_{s+m}$, where between $B^{n-1}_{s-1}$ and $B^{n+k}_{s+m}$ there are $k\geq2$ letters $Q$ or/and $m\geq2$ letters $T$, then
$a_{n+1}=\ldots=a_{n+k-1}=1, \, a_{n+k}\in\{1,2\}$ and $b_{s+1}=\ldots=b_{s+m-1}=1, \, b_{s+m}\in\{1,2\}$.
\end{proposition}
\begin{proof}
 We prove the statement only for the partial quotients of $\alpha$. The statement for partial quotients of $\beta$ will follow for reasons of symmetry.\\
We know that $a_{n+k}\in\{1,2\}$ from Proposition \ref{prop1}. It is left to show that $a_{n+1}=\ldots=a_{n+k-1}=1$.
Suppose that this is not true, i.e. that there exist $i_0\in\{1,\ldots k-1\}$, such that $a_{n+i_0}=2$. Then by Lemma \ref{lemma3} we have either $\alpha_{n+i_0}\geq \sqrt2+1$, or $\alpha_{n+i_0+1}\geq \sqrt2+1$. Note that by the conditions of this lemma, there are letters $Q^{n+i_0-1}$ and $Q^{n+i_0}$ in the word $W$. So for one of the two options of $r=n+i_0-1$ or $r=n+i_0$ the conditions of Lemma \ref{lemma6} are satisfied and we can apply it. By Lemma \ref{lemma6} we find $x_0\in[q_{n+i_0-1}-1,q_{n+i_0}]$, such that
$$
 | \psi_\alpha(x_0) - \psi_\beta(x_0) | \geqslant (\sqrt{\sqrt2+1}-1) \cdot \min( \psi_\alpha(x_0), \psi_\beta(x_0) )=C_2\min( \psi_\alpha(x_0), \psi_\beta(x_0) ),
$$
which contradicts the assumption \eqref{eq1}. 
\end{proof}

Using lemmas and propositions above, we will show that under the assumption \eqref{eq1}, there can only be finitely many letters $B$ in the word $W$.\\
Indeed, assume that there are infinitely many letters $B$ in the word $W$. Let $B^{n-1}_{s-1} \dots B^{n+k}_{s+m}$ be an arbitrary subword, where between two consecutive letters $B$ there are $k$ letters $Q$ and $m$ letters $T$. By Proposition \ref{prop2} we already know that $a_{n+i} = 1$ for all $i = 1, \dots, k-1$ and $a_{n+k} \in \{1, 2\}$. The same holds for $b_j$ for $j = 1, \dots, s$. We denote 
\begin{align*}
k' = 
\begin{cases}
k, & a_{n+k} = 1, \\
k+1, & a_{n+k} = 2,
\end{cases}
&\qquad
m' = 
\begin{cases}
m, & b_{s+m} = 1, \\
m+1, & b_{s+m} = 2.
\end{cases}
\end{align*}
All possible cases are presented on Figure \ref{fig1} (without loss of generality, we assume that $k'\geq m'$, because we can always swap roles of $\alpha$ and $\beta$). We now proceed to the elimination of these cases.
\begin{figure}[h]
\begin{tikzpicture}
[
level 1/.style = {sibling distance = 2cm, level distance = 3cm},
level 2/.style = {sibling distance = 2cm, level distance = 4.3cm},
level 3/.style = {sibling distance = 2cm, level distance = 7.1cm},
every node/.append style = {draw},
grow=right,
edge from parent fork right
]

\node {$B^{n-1}_{s-1} \dots B^{n+k}_{s+m}$}
child {node [rectangle split, rectangle split parts=2, inner ysep=1pt]{
		\textbf{Case 2} \strut
		\nodepart{second}
		$k \geqslant 1, \ m \geqslant 1$ \strut
	}
	child {node [rectangle split, rectangle split parts=3, xshift=-0.15cm, yshift=-0.3cm]{
			\textbf{Case 2.4}
			\nodepart{second}
			$k' = m' + 1, \ k' \leqslant 3$
			\nodepart{third}
			$a_{n+k} = 1, \ b_{s+m} = 1$ 
		}
	}
	child {node [rectangle split, rectangle split parts=3]{
			\textbf{Case 2.3}
			\nodepart{second}
			$k' = m' +1, \ k' \leqslant 3$
			\nodepart{third}
			$a_{n+k} = 2$ or $b_{s+m} = 2$
		}
		child {node 
			[rectangle split, rectangle split parts=2, xshift=-1.7cm, inner ysep=1pt]{
				\textbf{Case 2.3.b} \strut
				\nodepart{second}
				$k' = 2, \ m' = 1$ \strut 
			}			
		}
		child {node [rectangle split, rectangle split parts=2, xshift=-1.7cm, inner ysep=1pt]{
				\textbf{Case 2.3.a} \strut
				\nodepart{second}
				$k' = 3, \ m' = 2$ \strut
			}	
		}
	}
	child {node [rectangle split, rectangle split parts=2, xshift=0.98cm, inner ysep=1pt]{
			\textbf{Case 2.2} \strut
			\nodepart{second}
			$k' > m' + 1$ or $k' = m' + 1, \ k' \geqslant 4$ \strut
		}
	}
	child {node [rectangle split, rectangle split parts=2, xshift=-0.9cm, yshift=-0.2cm, inner ysep=1pt]{
			\textbf{Case 2.1} \strut
			\nodepart{second}
			$k' = m'$ \strut
		}
	}
}
child {node [rectangle split, rectangle split parts=2, inner ysep=1pt]{
		\textbf{Case 1} \strut
		\nodepart{second}
		$k \geqslant 1, \ m = 0$ \strut
	}
};

\end{tikzpicture}
\caption{Partition into cases}
\label{fig1}
\end{figure}
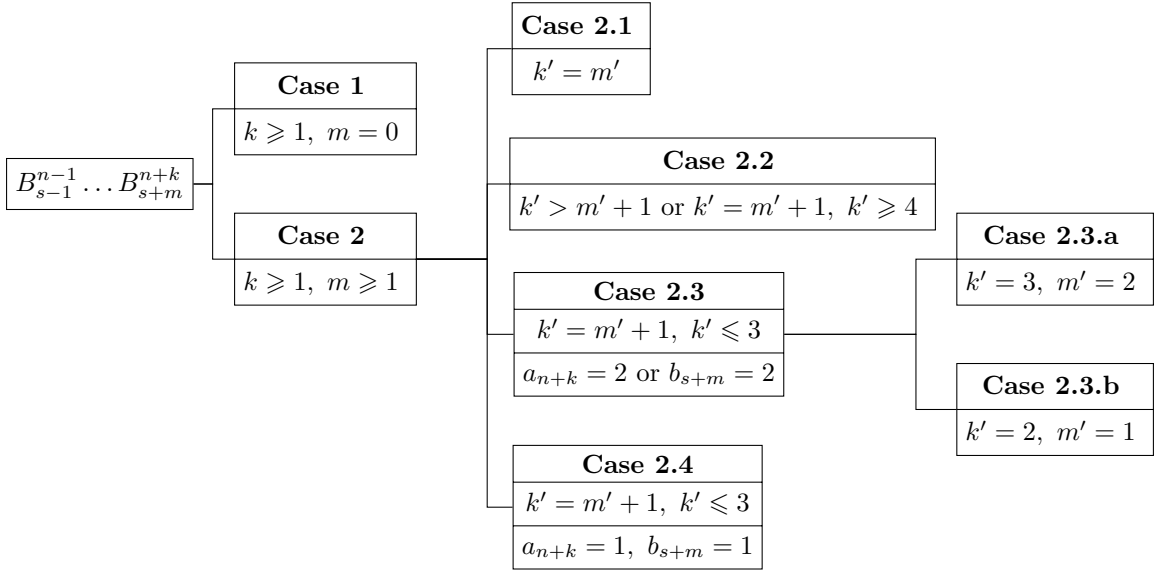

 \underline{Case 1:} Subwords of the form $B^{n-1}_{s-1}Q^n \ldots B^{n+k}_{s}$, i.e. there are only $k\geq1$ letters $Q^*$ between two consecutive letters $B$. 
 If $k=1$, then from Lemma \ref{lemma2} we know $a_{n+1}=2$. If $k\geq2$, then from Lemma \ref{prop2}, we know that $a_{n+1}=1$ and $a_{n+i}\leq 2$ for $i=2,\ldots k$. Combining all the information in this case, we have
 
 \begin{gather}
 \label{eq7}
 t_{s-1} = q_{n-1}, \\
 \label{eq8}
 t_s = q_{n+k} = a_{n+k} q_{n+k-1} + q_{n+k-2}=
 \begin{cases}
 2q_n+q_{n-1}, & k =1, \\
 (a_{n+k}a_{n+k-1}+1)q_{n+k-2}+a_{n+k}q_{n+k-3}\geq 2q_n+q_{n-1} , & k\geq2,
 \end{cases} \\
 \label{eq9}
 \frac{q_n}{q_{n-1}} > 1, \\
 \label{eq10}
 \alpha_{n+1}=a_{n+1}+\frac{1}{a_{n+2}+\ldots} > 
 \begin{cases}
 a_{n+1}=2 >  \frac{4}{3}, & k =1, \\
 a_{n+1}+\frac{1}{a_{n+2}+1}>1+\frac{1}{3}=\frac{4}{3} , & k\geq2.
 \end{cases}
 \end{gather}
As for $k=1$ we have $a_{n+1} = 2$, it follows from Proposition \ref{prop1} that  $\xi_{n-1} > \eta_{s-1}$. If $k > 1$, the same follows from Lemma \ref{lemma4}, so under the assumption (\ref{eq1}) we have
 \begin{equation*}
 \xi_{n-1} < K \eta_{s-1},
 \end{equation*}
 or, by (\ref{defal}),
 \begin{equation*}
 t_s + \frac{t_{s-1}}{\beta_{s+1}} < K (q_n + \frac{q_{n-1}}{\alpha_{n+1}}).
 \end{equation*}
 From (\ref{eq7}), (\ref{eq8}), (\ref{eq9}) and (\ref{eq10}) we get
 \begin{equation*}
 \frac{1}{\beta_{s+1}} < (K - 2) \frac{q_n}{q_{n-1}} + \frac{K}{\alpha_{n+1}} - 1 < (K-2) +\frac{3K}{4}-1< 0,
 \end{equation*}
 which is a contradiction.

\underline{Case 2:}
 Subwords of the form $B^{n-1}_{s-1}\dots B^{n+k}_{s+m}$, where between two consecutive letters $B^{n-1}_{s-1}$ and $B^{n+k}_{s+m}$ there are $k \geq 1$ letters $Q^*$ and $m \geq 1$ letters $T_*$ in any order.

We know that 
\begin{equation}
	\label{eq2}
	\langle a_1, \ldots , a_{n-1} \rangle =  q_{n-1} = t_{s-1} = \langle b_1, \ldots , b_{s-1} \rangle 
\end{equation}
and
\begin{equation*}
	\langle a_1, \ldots, a_{n+k} \rangle = q_{n+k} = t_{s+m} =  \langle b_1, \ldots, b_{s+m} \rangle.
\end{equation*}
By Proposition \ref{prop2} either $a_{n+i} = 1$ for all $i = 1, \dots, k$ or $a_{n+i} = 1$ for $i = 1, \dots, k-1$ and $a_{n+k} = 2$. The same holds for $b_{s+j}, j=1,\ldots,m$. 

Denote by $\langle a_1, \ldots, a_n \rangle $ the denominator $q_n$ of the continued fraction $[0;a_1,\ldots, a_n]$.
Note that by the recursive formula for denominators, we have
$$
  \langle a_1, \ldots, a_n,2 \rangle =2 \langle a_1, \ldots, a_n \rangle +  \langle a_1, \ldots, a_{n-1} \rangle =    \langle a_1, \ldots, a_n, 1 \rangle +  \langle a_1, \ldots, a_n \rangle = \langle a_1, \ldots, a_n, 1,1 \rangle .
$$
So we can write denominator $q_{n+k}$ as
\begin{equation*}
	q_{n+k} = \langle a_1, \ldots, a_n, \underbrace{1, \ldots, 1}_{k'\ge1 \text{ times}} \rangle 
\end{equation*}
and, similary,
\begin{equation*}
	t_{s+m} = \langle b_1, \ldots, b_s, \underbrace{1, \ldots, 1}_{m'\ge1 \text{ times}}  \rangle .
\end{equation*}
It is a well-known fact that
\begin{equation*}
	 \langle u_1,\ldots, u_i, v_1,\ldots, v_j \rangle = \langle u_1,\ldots u_i \rangle \langle v_1\ldots v_j \rangle + \langle u_1,\ldots u_{i-1} \rangle \langle v_2\ldots v_j \rangle,
\end{equation*}
so applying it to $q_{n+k} = t_{s+m}$ we get
\begin{equation*}
\langle a_1, \ldots a_{n} \rangle \langle \underbrace{1, \ldots, 1}_{k' \text{ times}} \rangle + 
\langle a_1, \ldots a_{n-1} \rangle \langle \underbrace{1, \ldots, 1}_{k'-1 \text{ times}} \rangle 
=\langle b_1, \ldots b_{s} \rangle \langle \underbrace{1, \ldots, 1}_{m' \text{ times}} \rangle + 
\langle b_1, \ldots b_{s-1} \rangle \langle \underbrace{1, \ldots, 1}_{m'-1 \text{ times}} \rangle,
\end{equation*}
or, for brevity,

\begin{equation*}
	q_{n}F_{k'+1} + q_{n-1} F_{k'} = t_{s} F_{m'+1} + t_{s-1} F_{m'},
\end{equation*}
where $F_n$ is the $n$th Fibonacci number, that is 
$F_{1} = F_2=1, \ F_{n+1}=F_n+F_{n-1}.$

From this equality and the fact that $t_{s-1} = q_{n-1}$ we can express $t_s$ as
\begin{equation}
	\label{eq3}
	t_s = \frac{F_{k'+1}}{F_{m'+1}}q_n + q_{n-1} \frac{F_{k'} - F_{m'}}{F_{m'+1}}.
\end{equation}

If $k\geq2$ and $k' \geq m'+1$, then $q_{n+1} \neq q_{n+k}$
 and we have
 \begin{equation}
 \label{qqt}
 q_{n+1} < t_s.
 \end{equation}
  Indeed,
\begin{multline*}
F_{m'+1} (t_s - q_{n+1})= F_{k'+1}q_n + q_{n-1} (F_{k'} - F_{m'})  - F_{m'+1} q_n - F_{m'+1} q_{n-1}   \\
= q_n (F_{k'+1} - F_{m'+1}) + q_{n-1} (F_{k'} - F_{m'} - F_{m'+1})  \\
 \geq q_n (F_{m'+2} - F_{m'+1}) + q_{n-1} (F_{m'+1} - F_{m'} - F_{m'+1})  = q_n F_{m'} -  q_{n-1} F_{m'} > 0,
\end{multline*}
and (\ref{qqt}) is proven.

Now we consider several subcases.

\underline{Case 2.1:} $k' = m'$. 

By (\ref{eq3}) we have $t_s = q_n$ and together with $t_{s-1}=q_{n-1}$ it is a contradiction with Lemma \ref{lemma1}.

\underline{Case 2.2:} $k' > m' + 1$ or $k' = m'+1$ with $ k' \geq 4$.

By (\ref{qqt}) we have a subword $B^{n-1}_{s-1} Q^{n}Q^{n+1} \dots B^{n+k}_{s+m}$ and by Lemma \ref{lemma4} we have $\xi_{n-1} > \eta_{s-1}$.

By assumption (\ref{eq1}) we have
\begin{equation*}
	\xi_{n-1} - \eta_{s-1} < C_2 \eta_{s-1},
\end{equation*}
or
\begin{equation*}
	\xi_{n-1} < K \eta_{s-1}.
\end{equation*}
 By (\ref{defal}) we can rewrite it as
\begin{equation*}
	t_s + \frac{t_{s-1}}{\beta_{s+1}} < K (q_n + \frac{q_{n-1}}{\alpha_{n+1}}),
\end{equation*}
and by (\ref{eq2}) and (\ref{eq3}) we get
\begin{equation}
	\label{eq5}
	\frac{1}{\beta_{s+1}} < \left(K - \frac{F_{k'+1}}{F_{m'+1}} \right) \frac{q_n}{q_{n-1}} + \frac{K}{\alpha_{n+1}} - \frac{F_{k'} - F_{m'}}{F_{m'+1}}.
\end{equation}

In this case we have 
\begin{equation}
\label{fib_bound1}
	\frac{F_{k'+1}}{F_{m'+1}} \geqslant
	\begin{cases}
		8/5, & k' = m' + 1,  \ k' > 4, \\
		5/3, & k' = m' + 1,  \ k' = 4, \\
		5/2, & k' > m' + 1,
	\end{cases}
\end{equation}
and
\begin{equation}
\label{fib_bound2}
	\frac{F_{k'} - F_{m'}}{F_{m'+1}} \geqslant
	\begin{cases}
		3/8,  &  k' = m' + 1, \ k' > 4, \\
		1/3, & k' = m' + 1, \ k' = 4, \\
		1, & k' > m' + 1.
	\end{cases}
\end{equation}

In the case $k' = m' + 1, \ k' > 4$ we have $a_{n+1} = a_{n+2} = a_{n+3} = 1$ and $a_{n+4} = 1$  or $a_{n+4} = 2$, and $b_{s+1} = b_{s+2} = 1$ and $b_{s+3} = 1$  or $b_{s+3} = 2$, so
\begin{equation}
	\label{bound1}
	\alpha_{n+1} = 1 + \cfrac{1}{1 + \cfrac{1}{1 + \cfrac{1}{a_{n+4} + \dots}}} > 1 + \cfrac{1}{1 + \cfrac{1}{1 + \cfrac{1}{2 + 1}}} = \frac{11}{7},
\end{equation}
\begin{equation}
\label{bound2}
\beta_{s+1} = 1 + \cfrac{1}{1 + \cfrac{1}{b_{s+3} + \dots}} < 1 + \cfrac{1}{1 + \cfrac{1}{2 + 1}} = \frac{7}{4}.
\end{equation}

In the case $k' = m' + 1, \ k' = 4$ there are two possible combinations for partial quotients of  $\alpha$: 
\begin{enumerate}
	\item $a_{n+1} = a_{n+2} = a_{n+3} = a_{n+4} = 1$,
	\item $a_{n+1} = a_{n+2} = 1, \ a_{n+3} = 2$,
\end{enumerate}
and two possible combinations for partial quotients of $\beta$: 
\begin{enumerate}
	\item $b_{s+1} = b_{s+2} = b_{s+3} = 1$,
	\item ${b_{s+1} = 1,\ b_{s+2} = 2}$.
\end{enumerate}
In the same manner it is easy to check that for these cases boundaries (\ref{bound1}) and (\ref{bound2}) are also true, so
\begin{equation}
	\label{bound3}
	\frac{1}{\alpha_{n+1}} < 
	\begin{cases}
		7/11, & k' = m' + 1,  \ k' \geq 4, \\
		1, & k' > m' + 1, \\
	\end{cases}
\end{equation}
\begin{equation}
\label{eq4}
\frac{1}{\beta_{s+1}} > 
\begin{cases}
4/7, & k' = m' + 1,  \ k' \geq 4, \\
0, & k' > m' + 1.
\end{cases}
\end{equation}

Substituting (\ref{fib_bound1}), (\ref{fib_bound2}), (\ref{bound3}) and $\frac{q_n}{q_{n-1}} > 1$ in (\ref{eq5}) we get
\begin{equation*}
	\frac{1}{\beta_{s+1}} < 
	\begin{cases}
		0.568, & k' = m' + 1,  \ k' > 4, \\
		 0.55, & k' = m' + 1,  \ k' = 4, \\
		 0,  & k' > m' + 1,
	\end{cases}	
\end{equation*}
which contradicts (\ref{eq4}).

\underline{Case 2.3:} $k'=m'+1, \ k' \leq 3$ with either $a_{n+k}=2$ or $b_{s+m}=2$. 

We consider several subcases here.

\underline{Case 2.3.a:} $k' = 3, \ m' = 2$. 

Under the conditions of this case we have three possible combinations for partial quotiens: 
\begin{enumerate}
	\item $a_{n+1} = a_{n+2} = a_{n+3} = 1, \ b_{s+1} = 2,$
	\item $a_{n+1} =1, \ a_{n+2} = 2, \ b_{s+1} = 2,$
	\item $a_{n+1} = 1, \ a_{n+2} = 2, \ b_{s+1} = b_{s+2} = 1$.
\end{enumerate}

By (\ref{eq3}) we have
\begin{equation}
\label{eq17}
q_{n+1} = q_n + q_{n-1} <
t_s = \frac{3}{2} q_{n} + \frac{1}{2}q_{n-1} < q_{n+2}.
\end{equation}

In the first two cases $b_{s+1} = 2$, and in the third $a_{n+2} = 2$, so by Proposition \ref{prop1} we have that $\eta_{s-1} > \xi_{n+1}$. 
Hence in all cases by assumption (\ref{eq1}) we have
\begin{equation*}
	\eta_{s-1} < K \xi_{n+1}.
\end{equation*}
or by (\ref{defal}) and (\ref{eq17}) we get
\begin{gather}
	\label{eq16}
	\alpha_{n+2} < \frac{(3/2 K - 1) q_n + \left( K / \beta_{s+1} + K / 2\right)q_{n-1}}{q_{n} + q_{n-1}}.
\end{gather}
Under this case conditions we have
\begin{equation}
\label{eq15}
\alpha_{n+2} >
\begin{cases}
	 3/2, & a_{n+2} = 1 \ a_{n+3} = 1, \\
	 2, & a_{n+2} = 2,
\end{cases}
\end{equation}
and
\begin{equation*}
\beta_{s+1} >
\begin{cases}
2, & b_{s+1} = 2, \\
 3/2, & b_{s+1} = b_{s+2} = 1.
\end{cases}
\end{equation*}
Substituting these estimations in (\ref{eq16}) we get
\begin{equation*}
\alpha_{n+2} < 
\begin{cases}
 3/2, & b_{s+1} = 2, \\
 2, & b_{s+1} = b_{s+2} = 1,
\end{cases}
\end{equation*}
which contradicts (\ref{eq15}).

\underline{Case 2.3.b:} $k' = 2, \ m' = 1$. Since $m'= 1$, we have $m=1$ and therefore $b_{s+1} = 1$. So by the assumption of case 2.3 we have $a_{n+1} = 2, \ k = 1$ and so
\begin{align}
	\label{eq11}
	\alpha_{n+1} &> a_{n+1} = 2, \\
	\label{eq6}
	\beta_{s+1} &< b_{s+1} + 1 = 2.
\end{align}

By (\ref{eq3}) we get
\begin{equation}
\label{eq14}
	t_s = 2q_n,
\end{equation}
hence we have a subword $B^{n-1}_{s-1}Q^n T_s B^{n+1}_{s+1}$ and
since $a_{n+1} = 2, \ q_n < t_s$ by Proposition \ref{prop1} we have $\xi_{n-1} > \eta_{s-1}$ and as in the case 2.2 we can rewrite it as
\begin{equation*}
t_s + \frac{t_{s-1}}{\beta_{s+1}} < K (q_n + \frac{q_{n-1}}{\alpha_{n+1}}),
\end{equation*}
or by (\ref{eq14}) and (\ref{eq11})  we get
\begin{equation*}
	\frac{1}{\beta_{s+1}} < \left(K - 2 \right) \frac{q_n}{q_{n-1}} + \frac{K}{\alpha_{n+1}} < 0.34,
\end{equation*}
which contradicts (\ref{eq6}).

\underline{Case 2.4:} $k' = m'+1, \ k' \leq 3$ with both $a_{n+k}=1$ and $b_{s+m}=1$.\\
In this case we will use the fact that we already excluded all possible subwords of the form $B\ldots B$, except for a couple of options. We list all (up to the swap of denominators of $\alpha$ and $\beta$ as for now we do not distinguish them) cases which are left.
\begin{enumerate}
	\item  $B_{s-1}^{n-1}\ldots B_{s+2}^{n+3}$ with  $(a_{n+1}, a_{n+2}, a_{n+3}) = (1,1,1)$ and $(b_{s+1}, b_{s+2}) = (1,1)$;
	\item  $B_{s-1}^{n-1}\ldots B_{s+1}^{n+2}$ with  $(a_{n+1}, a_{n+2}) = (1,1)$ and $b_{s+1} = 1$.
\end{enumerate}
By \eqref{qqt} and the fact that all partial quotients after letters $Q$ and $T$ are equal to $1$, we get that the subwords in each case look like
\begin{enumerate}
	\item  $B_{s-1}^{n-1}Q Q T T Q B_{s+2}^{n+3}$;
	\item  $B_{s-1}^{n-1}Q Q T B_{s+1}^{n+2}$
\end{enumerate}
respectively, where we omitted indices of the letters $Q$ and $T$.

We will show that both subwords $B_{s-1}^{n-1}Q Q T B_{s+1}^{n+2}$ and $B_{s-1}^{n-1}Q Q T T Q B_{s+2}^{n+3}$ can not occur. Let us deal with the case $B_{s-1}^{n-1}Q Q T B_{s+1}^{n+2}$ first. By \eqref{eq3} we have
$$
t_s = 2q_n
$$
and by Lemma \ref{lemma4} we have $\xi_{n-1}>\eta_{s-1}$. Note that we always have $b_{s+3}=1$ as there will be a letter $T_{s+2}$ in any of the four (here we took into consideration swap of denominators of $\alpha$ and $\beta$)  potential options of subwords after the subword $B_{s-1}^{n-1}Q Q T B_{s+1}^{n+2}$. Using this and $b_{s+2}\geq1$ we come to
$$
\beta_{s+1} = b_{s+1} + \cfrac{1}{b_{s+2} + \cfrac{1}{b_{s+3} + \dots}} < 1 + \cfrac{1}{1 + \cfrac{1}{1 + 1}} = \frac{5}{3}.
$$
We also have
$ \alpha_{n+1}>3/2$ and $\frac{q_n}{q_{n-1}}>1$. As before, we get
\begin{equation*}
\frac{3}{5}<\frac{1}{\beta_{s+1}} < \left(K - 2 \right) \frac{q_n}{q_{n-1}} + \frac{K}{\alpha_{n+1}} < (K-2)+\frac{K}{3/2}<0.59,
\end{equation*}
which is a contradiction.  \\

For the subword  $B_{s-1}^{n-1}Q Q T T Q B_{s+2}^{n+3}$, by \eqref{eq3} we have
$$
t_s = \frac{3}{2} q_{n} + \frac{1}{2}q_{n-1}.
$$
By Lemma \ref{lemma4} we get $\eta_{s-1} > \xi_{n+1}$. As before, we come to 
\begin{equation*}
	\frac{1}{t_s + \frac{t_{s-1}}{\beta_{s+1}}} < \frac{K}{q_{n+1}\alpha_{n+2} + q_n},
\end{equation*}
or equivalently to
\begin{equation}\label{eq244}
(q_n + q_{n-1}) \alpha_{n+2} + q_n < K \left( \frac{3}{2} q_{n} + \frac{1}{2}q_{n-1} \right) + K \frac{q_{n-1}}{\beta_{s+1}}.
\end{equation}
Note that we always have $b_{s+4}=1$ as there will be a letter $T_{s+3}$ in any of the four (here we took into consideration swap of denominators of $\alpha$ and $\beta$) potential options of subwords after the subword  $B_{s-1}^{n-1}Q Q T T Q B_{s+2}^{n+3}$. Using this and $b_{s+3}\geq1$ we get
$$
\beta_{s+1} = b_{s+1} + \cfrac{1}{b_{s+2} + \cfrac{1}{b_{s+3} + \cfrac{1}{b_{s+4} +\dots}}} > 1 + \cfrac{1}{1 + \cfrac{1}{1  + \cfrac{1}{1 + 1}}} = \frac{8}{5}.
$$
By the same reason we will always have $a_{n+5}=1$, as there always will be a letter $Q^{n+4}$ in the next subword. Using this and the fact that $a_{n+2}=a_{n+3}=1, a_{n+4}\geq1$, we can bound
$$
\alpha_{n+2} = a_{n+2} + \cfrac{1}{a_{n+3} + \cfrac{1}{a_{n+4} + \cfrac{1}{a_{n+5} +\dots}}} > 1 + \cfrac{1}{1 + \cfrac{1}{1  + \cfrac{1}{1 + 1}}} = \frac{8}{5}.
$$ 
Substituting these bounds on tails of continued fractions into \eqref{eq244}, we come to
$$
1.6<\alpha_{n+2}<\frac{q_n (\frac{3K}{2}-1) +q_{n-1}\left(\frac{K}{2}+\frac{K}{\beta_{s+1}}\right)}{q_n+q_{n-1}}<\frac{1.331q_n+1.75q_{n-1}}{q_n+q_{n-1}}<\frac{1.531q_n+1.55q_{n-1}}{q_n+q_{n-1}}<1.55,
$$
which is a contradiction.

We came to a contradiction in every possible case of subwords of the form $B\ldots B$, so under the assumption \eqref{eq1} there are only finitely many letters $B$ in the word $W$. 

We can now finalize the proof of Theorem \ref{maintheorem}. Recall that under the assumption that the statement of this theorem is not true, we are deducing a contradiction.
By Proposition \ref{prop1} all partial quotients for $\alpha$ and $\beta$ are equal to $1$ or $2$ for large enough indices. By the conditions of the theorem we also know that one of the numbers, say $\alpha$, is not equivalent to $\frac{\sqrt5+1}{2}$. Hence there are infinitely many partial quotients $a_n=2$. By Lemma \ref{lemma3} it means that
 $$
\min ( \alpha_{n}, \alpha_{n+1} ) \geq \sqrt2+1.
$$
Take an index $j\in\{n,n+1\}$ for which $a_j\geq\sqrt2+1$ and apply Lemma \ref{lemma6} for $r=j$ to get a contradiction with  \eqref{eq1}. We came to a contradiction in all potential cases, hence the assumption was false. \qed

\section{Proof of Theorem \ref{optimality}}\label{sec4}
We want to show that the constant $C_2$ is optimal. Consider two irrational numbers, $\theta = \sqrt{2} + 1 = [2; 2, 2, \dots]$ and
$\omega$, which we will define below.

First, by classical Kronecker's approximation theorem in one-dimensional case, we know that for every $\varepsilon > 0$ there exist $U,V \in\mathbb{Z}$, such that
\begin{equation*}
	\left | V + \frac{U}{\theta} - \sqrt{\theta} \right | < \varepsilon.
\end{equation*} 
Define the sequence $X_n$ as
\begin{equation*}
	X_0 = U, \,\,\, X_1 = V, \,\,\, X_{n+1} = 2X_n + X_{n-1} \text{\, for every\,\,\,} n\ge 1. 
\end{equation*}
Solving a linear recurrence relation, we get
\begin{equation*}
	X_n = A \theta^n + B (-\theta)^{-n}, \text{\quad where } A=\frac{V + U \theta^{-1}}{2 \sqrt{2}}\approx \frac{\sqrt\theta}{2\sqrt2}\approx0.549^+.
\end{equation*}
As $A>0$, there exists $k\in\mathbb{Z_+}$, such that $X_{k-1}, X_k \geqslant 1$. Let
\begin{equation*}
	\frac{X_{k-1}}{X_k} = [0; b_l,\ldots, b_1], \quad b_j \in\mathbb{Z_+}
\end{equation*}
and define $\omega$ as 
\begin{equation*}
	\omega = [0; b_1,\ldots, b_l, \overline{2} ].
\end{equation*}
We can show that for every $\varepsilon'>0$ the inequality
\begin{equation}\label{keps}
\Bigl | \psi_\theta(t) - \psi_\omega(t) \Bigl | \leqslant (C_2+\varepsilon') \cdot \min(\psi_{\theta}(t), \psi_{\omega}(t))
\end{equation}
holds for all $t$ large enough.\\
We know that denominator $Q_n$ of convergent $P_n/Q_n$ to $\theta$ is equal to
\begin{equation*}
	Q_n = \frac{1}{4} \Bigl( (2 + \sqrt{2}) \theta^n + (2 - \sqrt{2})(-\theta)^{-n} \Bigr).
\end{equation*} 
Denominator $s_n$ of convergent $r_n/s_n$ to $\omega$ is equal to
\begin{equation*}
s_n = X_{n-n_0}, \quad n_0 = k-l
\end{equation*}
for $n$ large enough. Note, that 
\begin{equation*}
\frac{Q_n}{X_n} \sim \frac{2 + \sqrt{2}}{4A} =\frac{1}{ (\sqrt{2} - 1) (V + U \theta^{-1})}\approx\sqrt\theta. 
\end{equation*} 
We use
\begin{equation*}
Q_{n-1}<X_n < Q_n  <  X_{n+1} 
\end{equation*} 
and 
\begin{gather*}
\psi_\theta(t) := \xi_n = \frac{1}{Q_{n+1} + \frac{Q_n}{\theta_{n+2}}} \sim \frac{1}{Q_n \left(\theta + \frac{1}{\theta} \right) } \sim \frac{1}{2 \sqrt{2} Q_n}, \quad Q_n \leqslant t < Q_{n+1}, \\
\psi_\omega(t) := \eta_n = \| s_{n+n_0} \eta \|
\sim \frac{1}{2 \sqrt{2} X_n}, \quad X_n \leqslant t < X_{n+1}, \\
\min(\psi_{\theta}(t), \psi_{\omega}(t)) = \begin{cases}
	\eta_n, &X_n \leqslant t < Q_n, \\
	\xi_{n}, & Q_n \leqslant t < X_{n+1}.
\end{cases}
\end{gather*}
Combining everything we get
\begin{equation*}
\frac{\xi_{n-1}}{\eta_n} \sim \frac{X_n}{Q_{n-1}} = \sqrt\theta + O(\varepsilon), \quad \frac{\eta_{n}}{\xi_n} \sim \frac{Q_n}{X_{n}} = \sqrt\theta + O(\varepsilon).
\end{equation*} 
Latter is equivalent to 
$$
 \xi_{n-1} \sim \sqrt\theta\eta_n,  \,\,\,\,  \eta_n\sim \sqrt\theta \xi_n,
$$
meaning that for all $X_n \leqslant t < Q_n$ we have
$$
\Bigl | \psi_\theta(t) - \psi_\omega(t) \Bigl | = |  \xi_{n-1} -\eta_n  | \sim (\sqrt\theta-1)\eta_n = C_2 \eta_n
$$
and for all $Q_n \leqslant t < X_{n+1}$ we have
$$
\Bigl | \psi_\theta(t) - \psi_\omega(t) \Bigl | = | \eta_n - \xi_n | \sim (\sqrt\theta-1)\xi_n = C_2 \xi_n.
$$
Hence (\ref{keps}) is proven.
\qed

\section{Final remarks}\label{sec5}
In the previous section, we showed that there exists a pair of numbers $\theta\sim_c\omega\sim_c\sqrt2$, such that $C_{\theta,\omega}=C_2$. However, as we will show next, this is not true for all pairs of numbers equivalent to $\sqrt2$.
\begin{remark}\label{remark1}
\textit{For each constant $C_2'\in[\sqrt{\sqrt2+1}-1,\sqrt2)$, there exists a pair of numbers $(\theta,\omega)$ with $\theta\pm\omega\notin\Z$, such that $\theta\sim_c\omega\sim_c\sqrt2$ and 
$$
C_{\theta,\omega}=C_2'.
$$
 }
\end{remark}
\begin{proof}
Construction from the proof of Theorem \ref{optimality} can be generalized in a following way. Fix $0<x<1$. We build $\omega=\omega(x)\sim_c\sqrt2$  using the following procedure: for any $\varepsilon>0$ there exists $U,V \in\mathbb{Z}$, such that
\begin{equation*}
	\left | V + \frac{U}{\theta} - \theta^{x} \right| < \varepsilon.
\end{equation*} 
Proceeding as in the in the proof of Theorem \ref{optimality}, one comes to
$$
\frac{Q_n}{X_n} \sim\theta^{1-x}. 
$$
Finally, we will get two different types of distances between $\psi_\omega(t)$ and $\psi_\theta(t)$. Namely, for all $X_n \leqslant t < Q_n$ we have
$$
\Bigl | \psi_\theta(t) - \psi_\omega(t) \Bigl | = |  \xi_{n-1} -\eta_n  | \sim (\theta^x-1)\eta_n,
$$
and for all $Q_n \leqslant t < X_{n+1}$ we have
$$
\Bigl | \psi_\theta(t) - \psi_\omega(t) \Bigl | = | \eta_n - \xi_n | \sim (\theta^{1-x}-1)\xi_n.
$$
This means that for $\theta$ and $\omega(x)$ there are infinitely many $t\in\N$, for which
$$
\Bigl | \psi_\theta(t) - \psi_\omega(t) \Bigl | \sim C(x) \min(\psi_\theta(t),\psi_\omega(t)),
$$
where $C(x)=\max(\theta^{x},\theta^{1-x})-1$. Moreover, for all $\epsilon>0$, for large enough $t$ we have
$$
\Bigl | \psi_\theta(t) - \psi_\omega(t) \Bigl | < (C(x)+\epsilon) \min(\psi_\theta(t),\psi_\omega(t)).
$$
As $x\in(0,1)$, we get statement of this remark.
\end{proof}
In the case $\alpha\sim_c\beta\sim_c\frac{\sqrt5+1}{2}=\tau$ we can do the process as in Remark \ref{remark1}. Omitting the details, we formulate a final statement
\begin{remark}\label{remark2}
\textit{For each constant $C_1'\in[\sqrt\tau-1,\tau-1)$, there exists a pair of numbers $(\alpha,\beta)$ with $\alpha\pm\beta\notin\Z$, such that $\alpha\sim_c\beta\sim_c\tau$ 
and 
$$
C_{\alpha,\beta}=C_1'.
$$
 }
\end{remark}
As $\tau-1>\sqrt{\sqrt2+1}-1$, from Remark \ref{remark1} and Remark \ref{remark2} we see that there exists a pair $\alpha\sim_c\beta\sim_c\tau$ and a pair $\theta\sim_c\omega\sim_c\sqrt2$, such that
$$
C_{\alpha,\beta} > C_{\theta,\omega}.
$$
Surprisingly, we also have the following statement.
\begin{remark}\label{remark3}
For any pair $(\alpha,\beta)$, such that $\alpha\sim_c\tau$ and $\beta\sim_c\sqrt2$ we have $C_{\alpha,\beta}>C_{\alpha',\beta'}$ for any pair $\alpha'\sim_c \beta'\sim_c \tau$. 


\end{remark}

\begin{proof}
We know that $\alpha\sim_c\tau$ and $\beta\sim_c\theta$. First, we can show that there are infinitely many subwords $QQ$ in a word $W$ for such $\alpha$ and $\beta$. Indeed, assume that $q_n \leqslant t_s$, then \begin{equation*}
q_{n+1} \sim \tau q_n < t_{s+1} \sim \theta t_s,
\end{equation*}
where $q_n, \, q_{n+1}$ and $t_s, \, t_{s+1}$ are denominators of convergents to $\tau$ and $\theta$. If $q_{n+2} < t_{s+1}$ then the claim is proven. So one can assume  
\begin{equation*}
	q_n \leqslant t_s < q_{n+1} < t_{s+1} < q_{n+2}.
\end{equation*}
Then we have
\begin{equation*}
q_{n+3} \sim \tau^3 q_n \approx 4.236 q_n < 5.83 t_s \approx  \theta^2 t_s \sim t_{s+2},
\end{equation*}
so $t_{s+1} < q_{n+2} < q_{n+3} < t_{s+2}$, and the claim is proven.

Now, consider a subword $T_{s}Q^{n}Q^{n+1}$ or $B_s^{n-1}Q^nQ^{n+1}$, which occurs infinitely many times by the fact that $QQ$ occurs infinitely many times. Then either $\xi_{n-1} > \xi_{n} \geqslant \eta_{s}$ or $\eta_{s} \geqslant \xi_{n} > \xi_{n+1}$. In the first case we have 
\begin{equation*}
	\xi_{n-1} - \eta_s \geqslant \xi_{n-1} - \xi_n = (\tau - 1) \xi_n > C_{\alpha', \beta'} \cdot \eta_s \  \text{ for all } \ \alpha' \sim_c \beta'\sim_c \tau,
\end{equation*}
where the last inequality comes from the statement of Remark \ref{remark2}. Case $\eta_{s} \geqslant \xi_{n} > \xi_{n+1}$ follows similar steps.

\end{proof}

\subsection*{Acknowledgements}
This research began during the ``Diophantine Analysis, Dynamics and Related Topics'' conference in Technion, Israel. Both authors thank Nikolay Moshchevitin for organising this event and careful reading of earlier versions of the manuscript. Nikita Shulga thanks Mumtaz Hussain for funding his travel to this conference.

Second author is a scholarship holder of "BASIS" Foundation for Advancement of Theoretical Physics and Mathematics and was partially supported by the Russian Science Foundation, grant no. 22-41-02028.

\bibliographystyle{abbrv}

\end{document}